\documentclass[a4paper,reqno,12pt]{amsart}
\usepackage[T1]{fontenc}
\usepackage[utf8]{inputenc}
\usepackage[english]{babel}

\usepackage{amsmath}
\usepackage{amsthm}
\usepackage{amssymb}
\usepackage{mathtools}
\usepackage{amsmath}
\usepackage{amsfonts}
\usepackage{mathrsfs}
\usepackage{esint}
\usepackage{yfonts}
\usepackage{quoting}
\usepackage{graphics}
\usepackage{booktabs}
\usepackage{bm}
\usepackage{bbm}
\usepackage{lmodern}
\usepackage{stmaryrd}
\usepackage{caption}
\usepackage{enumitem}			
\usepackage{tikz}
\usepackage{bookmark}
\usepackage{pgfplots}
\usepgfplotslibrary{fillbetween}
\usepackage{orcidlink}
\usepackage[mathscr]{euscript}

\usepackage{hyperref}
\hypersetup{
	colorlinks=true,
	linkcolor=blue,
	citecolor=blue,
	urlcolor=black,
	linktoc=all
}

\usepackage[margin=1.1in]{geometry}

\quotingsetup{font=small}

\theoremstyle{plain}
\newtheorem{proposition}{Proposition}[section]

\theoremstyle{plain}
\newtheorem{theorem}{Theorem}[section]
\numberwithin{equation}{section}	

\newtheorem{theoremA}{Theorem}

\newtheorem{propositionE}[theoremA]{Proposition}

\theoremstyle{plain}

\theoremstyle{plain}
\newtheorem{corollary}{Corollary}[theorem]

\theoremstyle{definition}
\newtheorem{remark}{Remark}[section]

\DeclarePairedDelimiter{\abs}{\lvert}{\rvert}		

\DeclareMathOperator{\tr}{tr}				
\DeclareMathOperator{\Id}{Id_\textit{n}}	

\newcommand{\numberset}{\mathbb}
\newcommand{\N}{\numberset{N}}			
\newcommand{\R}{\numberset{R}}			
\newcommand{\past}{p^\ast}				
\newcommand{\loc}{{\rm loc}}			
\newcommand{\stressu}{{a\!\left(\nabla u\right)}}	
\newcommand{\stressv}{{a\!\left(\nabla v\right)}}	

\def\Xint#1{\mathchoice
	{\XXint\displaystyle\textstyle{#1}}%
	{\XXint\textstyle\scriptstyle{#1}}%
	{\XXint\scriptstyle\scriptscriptstyle{#1}}%
	{\XXint\scriptscriptstyle\scriptscriptstyle{#1}}%
	\!\int}
\def\XXint#1#2#3{{\setbox0=\hbox{$#1{#2#3}{\int}$ }
		\vcenter{\hbox{$#2#3$ }}\kern-.6\wd0}}

\def\dashint{\Xint-}

\usepackage{enumitem}

\allowdisplaybreaks

\begin{document}
	
	\title[Positive solutions to the critical~$p$-Laplace equation]{Classification results for bounded positive solutions to the critical~$p$-Laplace equation}

	\author[Giulio Ciraolo]{Giulio Ciraolo \orcidlink{0000-0002-9308-0147}}
	\address[]{Giulio Ciraolo. Dipartimento di Matematica ‘Federigo Enriques’, Università degli Studi di Milano, Via Cesare Saldini 50, 20133, Milan, Italy}
	\email{giulio.ciraolo@unimi.it}
	
	\author[Michele Gatti]{Michele Gatti \orcidlink{0009-0002-6686-9684}}
	\address[]{Michele Gatti. Dipartimento di Matematica ‘Federigo Enriques’, Università degli Studi di Milano, Via Cesare Saldini 50, 20133, Milan, Italy}
	\email{michele.gatti1@unimi.it}
	
	\date{\today}
	\dedicatory{}
	
	\begin{abstract}
		By providing optimal or nearly optimal integral estimates, we show that every positive, bounded or moderately growing, local weak solution to the critical~$p$-Laplace equation in~$\R^n$, with~$n\geq 3$, and whose infimum over a ball behaves properly must be a bubble.
	\end{abstract}
	
	\maketitle
	
	
	\section{Introduction}
	\label{sec:intro-class-bounded}
	
	We investigate the classification of positive solutions to the critical~$p$-Laplace equation
	\begin{equation}
		\label{eq:critica-plap}
		\Delta_p u + u^{\past-1} = 0 \quad\text{in } \R^n,
	\end{equation}
	where~$1<p<n$ and~$n \geq 3$. This is a classical topic in PDEs which arises from the characterization of critical points of the Sobolev inequality as well as from the Yamabe problem when~$p=2$.
	
	When~$p=2$, classical non-negative solutions to~\eqref{eq:critica-plap} were classified in the seminal works of Gidas, Ni \& Nirenberg~\cite{gnn}, Caffarelli, Gidas \& Spruck~\cite{cgs}, as well as Chen \& Li~\cite{cl}. This classification was later extended to the full range~$1<p<n$ for weak solutions in the energy space
	\begin{equation*}
		\mathcal{D}^{1,p}(\R^n) \coloneqq \{ u \in L^{\past}\!(\R^n) \mid \nabla u \in L^p(\R^n) \}
	\end{equation*}
	by Damascelli, Merch\'an, Montoro \& Sciunzi~\cite{dm}, V\'etois~\cite{vet}, and Sciunzi~\cite{sciu}, who proved that positive energy solutions to~\eqref{eq:critica-plap}, referred to as~$p$-\textit{bubbles}, are explicitly given by
	\begin{equation}
		\label{eq:pbubb-classif}
		U_p[z,\lambda] (x) \coloneqq \left( \frac{\lambda^\frac{1}{p-1} \, n^\frac{1}{p} \left(\frac{n-p}{p-1}\right)^{\!\!\frac{p-1}{p}}}{\lambda^\frac{p}{p-1}+\abs*{x-z}^\frac{p}{p-1}} \right)^{\!\!\frac{n-p}{p}} \!,
	\end{equation}
	where~$\lambda>0$ and~$z \in \R^n$ are the scaling and translation parameters, respectively. We also mention that, still in the case of energy solutions, Ciraolo, Figalli \& Roncoroni~\cite{cfr} addressed classification results in the anisotropic setting in convex cones of~$\R^n$. 
	
	The classification of solutions not assumed a priori to belong to the space~$\mathcal{D}^{1,p}(\R^n)$ is known only for~$p=2$, as established in~\cite{cgs,cl}. Moreover, a remarkable result due to Schoen -- see~\cite[Corollary~1.6]{li-zhang} -- implies that a positive solution~$u \in C^2(\R^n)$ to~\eqref{eq:critica-plap} for~$p=2$ enjoys the estimate
	\begin{equation*}
		\int_{B_R} \left(\abs*{\nabla u}^2 + u^{2^\ast}\right) dx \leq C_n,
	\end{equation*}
	for some dimensional constant~$C_n>0$, where~$B_R \coloneqq B_R(0)$. As a consequence, any positive solution~$u \in C^2(\R^n)$ to~\eqref{eq:critica-plap} immediately belongs to the energy space~$\mathcal{D}^{1,2}(\R^n)$.
	
	Therefore, the classification result for any~$1<p<n$ without the energy assumption remains an open problem. Specifically, one seeks to classify positive \textit{local weak solutions} to~\eqref{eq:critica-plap}, that is positive functions~$u \in W^{1,p}_{\loc}(\R^n) \cap L^\infty_{\loc}(\R^n)$ such that
	\begin{equation*}
		\int_{\R^n} \left\langle \abs*{\nabla u}^{p-2} \,\nabla u, \nabla \psi \right\rangle dx = 	\int_{\R^n} u^{\past-1} \psi \, dx \quad\text{for every } \psi \in C^\infty_c(\R^n).
	\end{equation*}
	Some progresses have been made in the last years. The first result in this direction was obtained by Catino, Monticelli \& Roncoroni~\cite{catino}. Their work was significantly improved by Ou~\cite{ou}, and later refined by V\'etois~\cite{vet-plap} and Sun \& Wang~\cite{sun-wang}. All these results can be summarized in the following theorem.
		
	\begin{theoremA}
	\label{th:class-noener}
		Let~$n \in \N$,~$1<p<n$, and let~$u \in W^{1,p}_{\loc}(\R^n) \cap L^\infty_{\loc}(\R^n)$ be a positive local weak solution to~\eqref{eq:critica-plap}. Moreover, assume that~$p_n < p < n$, where
		\begin{equation*}
			p_n \coloneqq
			\begin{cases}
				\begin{aligned}
					& \quad 1			&& \text{if } n=2, \\
					& \frac{n^2}{3n-2}	&& \text{if } n=3,4, \\
					& \frac{n^2+2}{3n}	&& \text{if } n \geq 5.
				\end{aligned}
			\end{cases}
		\end{equation*}
		Then,~$u$ is a~$p$-bubble of the form~\eqref{eq:pbubb-classif}.
	\end{theoremA}
	
	For~$p = 2$, since any solution~$u \in W^{1,2}_{\loc}(\R^n) \cap L^\infty_{\loc}(\R^n)$ to~\eqref{eq:critica-plap} is actually smooth by standard regularity theory, the conclusion of Theorem~\ref{th:class-noener} follows directly from the discussion above in every dimension. 

	In the general case~$1<p<n$, apart from Theorem~\ref{th:class-noener}, the classification of local weak solutions to~\eqref{eq:critica-plap} is known only under additional global decay or boundedness assumptions. Motivated by this limitation, we establish the following higher integrability result under the assumption that the infimum of the solution behaves properly at infinity, namely that there exists some constant~$C>0$ such that
		\begin{equation}
			\label{eq:ass-inf}
			\inf_{B_R} u \leq C R^{-{\frac{n-p}{p-1}}} \quad\text{for every } R \geq 1.
		\end{equation}
		In particular, assumption~\eqref{eq:ass-inf} is required at only one point in the proof.
	
	\begin{theorem}
		\label{th:u-Lpast-1}
		Let~$n \in \N$,~$1<p<n$, and let~$u \in W^{1,p}_{\loc}(\R^n) \cap L^\infty_{\loc}(\R^n)$ be a non-negative, non-trivial, local weak solution to~\eqref{eq:critica-plap} satisfying~\eqref{eq:ass-inf}. Then,~$u \in L^{\past-1}(\R^n)$.
	\end{theorem}
	
	As a consequence, we provide a complete classification under the sole assumption that the solution is bounded.

	\begin{theorem}
	\label{th:class-bounded}
		Let~$n \in \N$,~$1<p<n$, and let~$u \in W^{1,p}_{\loc}(\R^n) \cap L^\infty_{\loc}(\R^n)$ be a non-negative, non-trivial, local weak solution to~\eqref{eq:critica-plap} satisfying~\eqref{eq:ass-inf}. Suppose, in addition, that~$u$ is globally bounded, i.e.~$u \in L^\infty(\R^n)$. Then,~$u$ is a~$p$-bubble of the form~\eqref{eq:pbubb-classif}.
	\end{theorem}
 	
	Furthermore, Theorem~\ref{th:u-Lpast-1} motivates the following corollary.
	
	\begin{corollary}
		\label{cor:absolut-cont}
		Let~$n \in \N$,~$1<p<n$, and let~$u \in W^{1,p}_{\loc}(\R^n) \cap L^\infty_{\loc}(\R^n)$ be a non-negative, non-trivial, local weak solution to~\eqref{eq:critica-plap} satisfying~\eqref{eq:ass-inf}. Suppose, in addition, that~$u^{\past-1}$ is uniformly continuous. Then,~$u$ is a~$p$-bubble of the form~\eqref{eq:pbubb-classif}.
	\end{corollary}
	
	Actually, Theorem~\ref{th:class-bounded} is a particular case of the following more general statement.
	
	\begin{theorem}
		\label{th:class-Lq}
		Let~$n \in \N$,~$1<p<n$, and let~$u \in W^{1,p}_{\loc}(\R^n) \cap L^\infty_{\loc}(\R^n)$ be a non-negative, non-trivial, local weak solution to~\eqref{eq:critica-plap} satisfying~\eqref{eq:ass-inf}. Suppose, in addition, that~$u \in L^q(\R^n)$ for some~$q \in [\past,+\infty]$. Then,~$u$ is a~$p$-bubble of the form~\eqref{eq:pbubb-classif}.
	\end{theorem}

	We establish all these conclusions by deriving new optimal and nearly optimal integral estimates for solutions to~\eqref{eq:critica-plap} satisfying~\eqref{eq:ass-inf}. These estimates are optimal or nearly optimal in the sense that they are sharp or nearly sharp for the~$p$-bubbles~\eqref{eq:pbubb-classif}.
	
	Specifically, the key idea is the derivation of integral estimates for~$u^\gamma$ on the ball~$B_R$ in two different regimes:~$\gamma < p_\ast-1$ and~$p_\ast-1 \leq \gamma \leq \past-1$ -- see~\eqref{eq:def-ps-ast} below for the definitions of these exponents.
	
	The first estimate is, to some extend, analogous to the one established on the unit ball by Bidaut-Veron~\cite{bveron-ball}. Nevertheless, Kelvin transform-type arguments do not seem adequate to obtain it directly from~\cite{bveron-ball}. Instead, we derive it by combining the weak Harnack inequality with an estimate due to Shakerian \& V\'etois~\cite{sh-vet}. The estimates in the second regime are obtained via test function arguments as in~\cite{ou}, starting from the previous ones. In particular, we prove that~$u \in L^{\past-1}(\R^n)$, which immediately yields the desired results by interpolation.
	
	In fact, the boundedness assumption can be further relaxed to allow for moderate growth at infinity.
	
	\begin{theorem}
		\label{th:class-crescita}
		Let~$n \in \N$,~$1<p<n$, and let~$u \in W^{1,p}_{\loc}(\R^n) \cap L^\infty_{\loc}(\R^n)$ be a non-negative, non-trivial, local weak solution to~\eqref{eq:critica-plap} satisfying~\eqref{eq:ass-inf}. Suppose, in addition, that
		\begin{equation}
		\label{eq:cond-p-beta}
			1<p<\frac{n^2+2n}{3n+1} \quad\text{and}\quad \beta \coloneqq \frac{2n(n-p)}{n^2+2n-(3n+1)p}
		\end{equation}
		and that, for some~$C>0$, we have
		\begin{equation}
		\label{eq:growth-beta}
			u(x) \leq C \,\abs*{x}^{\beta} \quad\text{for every } \abs*{x} \geq 1.
		\end{equation}
		Then,~$u$ is a~$p$-bubble of the form~\eqref{eq:pbubb-classif}.
	\end{theorem}
	
	
	\subsection*{Structure of the paper.}
	
	In Section~\ref{sec:int-est}, we introduce some auxiliary functions that will be used throughout the manuscript and establish some integral estimates for~$u$. These include Theorem~\ref{th:u-Lpast-1}. In Section~\ref{sec:proof-12-13}, we prove Theorems~\ref{th:class-bounded},~\ref{th:class-Lq}, and~\ref{th:class-crescita} and Corollary~\ref{cor:absolut-cont}. Finally, in Appendix~\ref{sec:furt-int-est}, we derive other integral estimates involving the gradient.
	
	
	\section{Integral estimates for~$u$}
	\label{sec:int-est}
	
	Since~$u \in L^\infty_{\loc}(\R^n)$, it is well known~\cite{diben,lad-ur,tolk} that~$u \in C^{1,\alpha}_{\loc}(\R^n)$ for some~$\alpha \in (0,1)$. Moreover, since~$u \geq 0$ in~$\R^n$ is non-trivial, the strong maximum principle in~\cite{vaz} implies that~$u > 0$ in~$\R^n$.
	
	On top of this, we introduce the auxiliary function
	\begin{equation}
		\label{eq:defv-class}
		v \coloneqq u^{-\frac{p}{n-p}}.
	\end{equation}
	We also define the so-called~$P$-function given by
	\begin{equation}
	\label{eq:defPfunct-class}
		P \coloneqq n \,\frac{p-1}{p} v^{-1} \, \abs*{\nabla v}^p + \left(\frac{p}{n-p}\right)^{\! p-1} v^{-1}.
	\end{equation}
	This has been used in~\cite{ou,sun-wang,vet-plap} to prove the classification result of Theorem~\ref{th:class-noener}, in~\cite{cami} for the semilinear case on Riemannian manifolds, and also in~\cite{cg-plap} to tackle quantitative stability issues related to~\eqref{eq:critica-plap}.
	
	Let us define the critical set of~$u \in W^{1,p}_{\loc}(\R^n) \cap L^\infty_{\loc}(\R^n)$ as
	\begin{equation*}
		\mathcal{Z}_u \coloneqq \left\{x \in \R^n \,\lvert\, \nabla u(x)=0 \right\}\!,
	\end{equation*}
	and the stress field associated with~$u$ by
	\begin{equation*}
		\stressu \coloneqq \abs*{\nabla u}^{p-2} \,\nabla u,
	\end{equation*}
	which is extended to zero on~$\mathcal{Z}_u$. Then, thanks to the regularity results in~\cite{carlos} and~\cite{lou}, the set~$\mathcal{Z}_u$ is negligible, i.e.~$\abs*{\mathcal{Z}_u}=0$, and
	\begin{gather*}
		\notag
		\stressu \in W^{1,2}_{\loc}(\R^n), \\
		\label{eq:regu-classif}
		u \in W^{2,2}_{\loc}(\R^n \setminus \mathcal{Z}_u) \quad \mbox{and} \quad \abs*{\nabla u}^{p-2} \,\nabla^2 u \in L^{2}_{\loc}(\R^n \setminus \mathcal{Z}_u) \quad \mbox{for every } 1<p<n.
	\end{gather*}
	From its definition in~\eqref{eq:defv-class},~$v$ inherits some regularity properties from~$u$, in particular
	\begin{gather*}
		\notag
		\mathcal{Z}_v = \mathcal{Z}_u, \\
		\label{eq:regv-classif}
		v \in C^{1,\alpha}_{\loc}(\R^n) \cap C^{\infty} \!\left(\R^n \setminus \mathcal{Z}_v\right) \quad \mbox{and} \quad \stressv \in W^{1,2}_{\loc}(\R^n), \\
		\notag
		v \in W^{2,2}_{\loc}(\R^n \setminus \mathcal{Z}_v) \quad \mbox{and} \quad \abs*{\nabla v}^{p-2} \,\nabla^2 v \in L^{2}_{\loc}(\R^n \setminus \mathcal{Z}_v) \quad \mbox{for every } 1<p<n.
	\end{gather*}
	Moreover,~$v$ is a weak solution to
	\begin{equation}
	\label{eq:eq-for-v-classif}
		\Delta_{p} v = P \quad\text{in } \R^n.
	\end{equation}
	
	To keep the notation concise, we also define the couple of critical exponents associated with~$u$ as
	\begin{equation}
	\label{eq:def-ps-ast}
		p_\ast \coloneqq \frac{p(n-1)}{n-p} \quad\text{and}\quad \past \coloneqq \frac{np}{n-p},
	\end{equation}
	and the exponent related to~$v$ as
	\begin{equation}
		\label{eq:psharp}
		p_\sharp \coloneqq \frac{p-1}{p} n.
	\end{equation}
	In particular, observe that~$p_\ast$ and~$p_\sharp$ satisfy the relation
	\begin{equation*}
		p_\ast - 1 = \frac{p}{n-p} \, p_\sharp.
	\end{equation*}
	 
	Recall also that, since~$u$ is weakly~$p$-superharmonic, Lemma 2.3 in~\cite{serr-zou} implies that
	\begin{equation}
		\label{eq:bbu-serzou}
		u(x) \geq C \,\abs*{x}^{-\frac{n-p}{p-1}} \quad\text{for all } x \in \R^n \setminus B_1,
	\end{equation}
	for some~$C>0$ depending only on~$n$,~$p$, and~$\min_{\partial B_1} u$. Hence, exploiting the definition of~$v$ in~\eqref{eq:defv-class}, we deduce that
	\begin{equation}
	\label{eq:babv-serzou}
		v(x) \leq C \,\abs*{x}^{\frac{p}{p-1}} \quad\text{for all } x \in \R^n \setminus B_1,
	\end{equation}
	for some~$C>0$ depending on~$n$,~$p$, and~$\min_{\partial B_1} u$.
	
	We point out that assumption~\eqref{eq:ass-inf} and~\eqref{eq:bbu-serzou} imply that the infimum of~$u$ on~$B_R$ has the same decay rate of the fundamental solution. Moreover, it is easy to see that
	\begin{equation*}
		\inf_{B_R} u \leq C R^{-{\frac{n-p}{p}}} \quad\text{for every } R \geq 1.
	\end{equation*}
	
	The following result provides optimal integral estimates for a positive weak solution to~\eqref{eq:critica-plap} in the regime of small exponents. This is the only point where assumption~\eqref{eq:ass-inf} is required.
	
	\begin{proposition}
		\label{prop:int-uv-1}
		Let~$n \in \N$,~$1<p<n$, and let~$u \in W^{1,p}_{\loc}(\R^n) \cap L^\infty_{\loc}(\R^n)$ be a non-negative, non-trivial, local weak solution to~\eqref{eq:critica-plap} satisfying~\eqref{eq:ass-inf}. Then, there exists~$R_0 \geq 1$ such that for every~$R \geq R_0$ and for either~$E = B_R$ or~$E=A_{R,3R} \coloneqq B_{3R} \setminus B_R$, we have
		\begin{equation}
			\label{eq:int-u-gamma}
			\int_{E} u^\gamma \, dx \leq C R^{n-\frac{n-p}{p-1} \gamma} \quad\text{for every } \gamma \in \left(-\infty,p_\ast-1 \right)\!,
		\end{equation}
		where~$C>0$ is a constant independent of~$R$. Equivalently, it follows that
		\begin{equation*}
			\label{eq:int-v-q}
			\int_{E} v^{-q} \, dx \leq C R^{n-\frac{p}{p-1} q} \quad\text{for every } q \in \left(-\infty,p_\sharp \right)\!,
		\end{equation*}
		where~$v$ and~$p_\sharp$ are defined in~\eqref{eq:defv-class} and~\eqref{eq:psharp}, respectively.
	\end{proposition}
	\begin{proof}
		We begin by defining, for~$R>0$, the continuous function
		\begin{equation}
			\label{eq:def-uR}
			u_R(x) \coloneqq R^{\frac{n-p}{p-1}} u(Rx) \quad\text{for } x \in \R^n
		\end{equation}
		and the quantity
		\begin{equation*}
			\Gamma_{\! R}(u) \coloneqq \min_{\partial B_1} u_R.
		\end{equation*}
		From~\eqref{eq:bbu-serzou}, we clearly have that
		\begin{equation*}
			\liminf_{R \to +\infty} \Gamma_{\! R}(u) = \sigma >0.
		\end{equation*}
		Moreover, by assumption~\eqref{eq:ass-inf} we have that~$\sigma$ is finite and  by the argument of Step~4.2 in~\cite{sh-vet}, we obtain
		\begin{equation}
			\label{eq:lim-alpha}
			\lim_{R \to +\infty} \Gamma_{\! R}(u) = \sigma.
		\end{equation}
		
		Then, we easily see that~$u_R$ is weakly~$p$-superharmonic in~$\R^n$, therefore it satisfies the weak Harnack inequality -- see, for instance,~\cite[Theorem~7.1.2]{ps} -- and for every ball~$B_\rho(x) \subseteq \R^n$ we have  
		\begin{equation}
			\label{eq:wh-uR}
			\rho^{-\frac{n}{\gamma}} \left(\int_{B_{\rho}(x)} u_R^\gamma \, dy \right)^{\!\!\frac{1}{\gamma}} \leq C_\sharp \inf_{B_{2\rho}(x)} u_R \quad\text{for every } \gamma \in \left(0,p_\ast-1 \right)\!,
		\end{equation}
		where~$C_\sharp>0$ depends only on~$n$,~$p$, and~$\gamma$.
		
		Applying~\eqref{eq:wh-uR} with~$x=0$ and~$\rho=1$ entails that
		\begin{equation}
			\label{eq:wh-B1}
			\left(\int_{B_{1}} u_R^\gamma \, dx \right)^{\!\!\frac{1}{\gamma}} \leq C_\sharp \inf_{B_{2}} u_R \leq C_\sharp \min_{\partial B_{1}} u_R = C_\sharp \,\Gamma_{\! R}(u),
		\end{equation}
		for~$\gamma \in \left(0,p_\ast-1 \right)$. From~\eqref{eq:lim-alpha}, we deduce that there exists~$R_0 \geq 1$ such that
		\begin{equation*}
			\Gamma_{\! R}(u) \leq 2 \sigma \quad\text{for every } R \geq R_0,
		\end{equation*}
		thus, combining with~\eqref{eq:wh-B1}, we have
		\begin{equation*}
			\int_{B_{1}} u_R^\gamma \, dx \leq \left(2 C_\sharp \sigma\right)^{\gamma} \quad\text{for every } \gamma \in \left(0,p_\ast-1 \right)\!,
		\end{equation*}
		for every~$R \geq R_0$. Finally, recalling~\eqref{eq:def-uR} and changing variables, we infer the validity of~\eqref{eq:int-u-gamma} for~$E=B_R$ and~$\gamma >0$.
		
		To prove~\eqref{eq:int-u-gamma} for~$E=A_{R,3R}$, we observe that
		\begin{equation}
			\label{eq:wh-annulus}
			\left(\int_{A_{\frac12,\frac32}} u_R^\gamma \, dx \right)^{\!\!\frac{1}{\gamma}} \leq C \min_{\partial B_{1}} u_R = C \,\Gamma_{\! R}(u),
		\end{equation}
		for some~$C>0$. Indeed, one recovers the annulus by a family of balls~$\{B_1(x_i)\}_{i=0}^k$ with~$\abs*{x_i}=1$,~$\abs*{x_i-x_{i-1}} \leq 1$ for~$i=1,\dots,k$, and~$u_R(x_0) = \Gamma_{\! R}(u)$. Moreover,~$k$ is independent of~$R$. A standard chaining argument yields~\eqref{eq:wh-annulus}, from which, thanks to the previous reasoning, we get the desired conclusion.
		
		Finally, when~$\gamma \leq 0$, the conclusion follows directly from~\eqref{eq:bbu-serzou}.
	\end{proof}
	
	We now aim to show that a positive weak solution to~\eqref{eq:critica-plap} enjoying~\eqref{eq:ass-inf} satisfies nearly optimal estimates also in the regime of large exponents. In particular, these estimates show that~$u \in L^{p^*-1}(\R^n)$ and, hence, they include Theorem~\ref{th:u-Lpast-1}.
	
	\begin{proposition}
		\label{prop:int-uv-2}
		Let~$n \in \N$,~$1<p<n$, and let~$u \in W^{1,p}_{\loc}(\R^n) \cap L^\infty_{\loc}(\R^n)$ be a non-negative, non-trivial, local weak solution to~\eqref{eq:critica-plap} satisfying~\eqref{eq:ass-inf}. Moreover, let~$q \in [p_\sharp,p_\sharp+1)$. Then, exists~$\varepsilon_0 \in (0,1)$, depending only on~$n$,~$p$, and~$q$, such that for every~$R \geq R_0$ -- where~$R_0$ is defined in Proposition~\ref{prop:int-uv-1} -- and~$\varepsilon \in (0,\varepsilon_0)$, we have
		\begin{equation}
			\label{eq:int-v-q-large}
			\int_{B_R} v^{-q} \, dx \leq  C R^{\varepsilon} \quad\text{for every } q \in [p_\sharp,p_\sharp+1),
		\end{equation}
		where~$C>0$ is a constant independent of~$R$. Moreover, it holds
		\begin{equation}
			\label{eq:int-v-q-psharp-1}
			\int_{B_R} v^{-p_\sharp-1} \, dx \leq  C.
		\end{equation}
		As a consequence,~$u \in L^{\past-1}(\R^n)$.
	\end{proposition}
	\begin{proof}
		In the following calculations~$C>0$ will denote a constant independent of~$R$, which may vary from line to line.
		
		We first claim that
		\begin{equation}
			\label{eq:1-int-toprove}
			\int_{B_R} v^{-q} \,\abs*{\nabla v}^p \, dx + \int_{B_R} v^{-q} \, dx \leq C R^{n-\frac{p}{p-1}(q-1)} \quad\text{for every } q < p_\sharp+1.
		\end{equation}
		
		Indeed, let~$\theta>p$ be fixed and~$\eta \in C^\infty_c(\R^n)$ be a cut-off function such that~$0 \leq \eta\leq 1$ in~$\R^n$,~$\eta = 1$ in~$B_R$,~$\eta=0$ in~$\R^n \setminus B_{2R}$, and~$\abs*{\nabla \eta} \leq 1/R$ in~$A_{R,2R}$. Testing~\eqref{eq:eq-for-v-classif} with~$\psi=v^{1-q} \eta^\theta$ and applying Young's inequality yield
		\begin{multline*}
			\left(p_\sharp+1-q\right) \int_{\R^n} \eta^\theta v^{-q} \,\abs*{\nabla v}^p \, dx + \left(\frac{p}{n-p}\right)^{\! p-1} \int_{\R^n} \eta^\theta v^{-q} \, dx \\
			\leq \varepsilon \int_{\R^n} \eta^\theta v^{-q} \,\abs*{\nabla v}^p \, dx + \frac{C}{\varepsilon^{p-1} R^{p}} \int_{\R^n} \eta^{\theta-p} v^{p-q} \, dx,
		\end{multline*}
		for~$\varepsilon>0$. By choosing~$\varepsilon>0$ sufficiently small, and since~$q < p_\sharp+1$, we deduce that
		\begin{equation*}
			\int_{\R^n} \eta^\theta v^{-q} \,\abs*{\nabla v}^p \, dx + \int_{\R^n} \eta^\theta v^{-q} \, dx \leq \frac{C}{R^{p}} \int_{B_{2R}} v^{p-q} \, dx.
		\end{equation*}
		On the other hand, we also have~$q-p<p_\sharp$ and since~$R \geq R_0$, applying Proposition~\ref{prop:int-uv-1}, we deduce that
		\begin{equation*}
			\int_{B_R} v^{-q} \,\abs*{\nabla v}^p \, dx + \int_{B_R} v^{-q} \, dx \leq C R^{n-\frac{p}{p-1}(q-p)-p},
		\end{equation*}
		hence~\eqref{eq:1-int-toprove} follows.
		
		Let~$q \in [p_\sharp,p_\sharp+1)$ be fixed. There exists~$\varepsilon_0 \in (0,1)$, depending only on~$n$,~$p$, and~$q$, such that
		\begin{equation*}
			s \coloneqq p_\sharp - \varepsilon < q < p_\sharp+1 -\varepsilon \eqqcolon r
		\end{equation*}
		for every~$\varepsilon \in (0,\varepsilon_0)$. Hence, Proposition~\ref{prop:int-uv-1} and~\eqref{eq:1-int-toprove} yield
		\begin{equation*}
			\int_{B_R} v^{-s} \, dx \leq C R^{\frac{p}{p-1}\varepsilon} \quad\text{and}\quad \int_{B_R} v^{-r} \, dx \leq C R^{\frac{p}{p-1}\varepsilon},
		\end{equation*}
		respectively. Let~$\lambda \in [0,1]$ be such that
		\begin{equation*}
			\frac{1}{q} = \frac{\lambda}{s} + \frac{1-\lambda}{r},
		\end{equation*}
		thus, by interpolation, we conclude that
		\begin{align*}
			\int_{B_R} v^{-q} \, dx &\leq \left(\int_{B_R} v^{-s} \, dx\right)^{\!\frac{\lambda q}{s}} \left(\int_{B_R} v^{-r} \, dx\right)^{\!\frac{(1-\lambda) q}{r}} \\
			&\leq \left( C R^{\frac{p}{p-1}\varepsilon} \right)^{\!\frac{\lambda q}{s}} \left( C R^{\frac{p}{p-1}\varepsilon} \right)^{\!\frac{(1-\lambda) q}{r}} =  C R^{\frac{p}{p-1}\varepsilon}.
		\end{align*}
		Up to a redefinition of $\varepsilon_0>0$, this shows the validity of~\eqref{eq:int-v-q-large} and we are left with the case~$q = p_\sharp+1$.
		
		Testing~\eqref{eq:eq-for-v-classif} with~$\psi=v^{-p_\sharp}\eta$, we get 
		\begin{equation}
			\label{eq:psharp-1-1}
			\int_{\R^n} \eta v^{-p_\sharp-1} \, dx \leq \frac{C}{R} \int_{A_{R,2R}} v^{-p_\sharp} \,\abs*{\nabla v}^{p-1} \, dx \leq \frac{C}{R} \int_{B_{2R}} v^{-p_\sharp} \,\abs*{\nabla v}^{p-1} \, dx.
		\end{equation}
		Using H\"older inequality, it follows that
		\begin{equation}
			\label{eq:psharp-1-2}
			\begin{split}
				\int_{B_{2R}} v^{-p_\sharp} \,\abs*{\nabla v}^{p-1} \, dx &= \int_{B_{2R}} v^{(\varepsilon-p_\sharp-1)\frac{p-1}{p}} \,\abs*{\nabla v}^{p-1} \, v^{-p_\sharp-(\varepsilon-p_\sharp-1)\frac{p-1}{p}} \, dx \\
				&\leq \left( \int_{B_{2R}} v^{\varepsilon-p_\sharp-1} \,\abs*{\nabla v}^{p} \, dx \right)^{\!\!\frac{p-1}{p}} \left( \int_{B_{2R}} v^{-p_\sharp+(1-\varepsilon)(p-1)} \, dx \right)^{\!\!\frac{1}{p}}\!.
			\end{split}
		\end{equation}
		Choosing
		\begin{equation*}
			0 < \varepsilon < \min\left\{1,\frac{n-p}{p}(p-1)\right\}\!,
		\end{equation*}
		we have that
		\begin{equation*}
			p_\sharp+1-\varepsilon < p_\sharp+1 \quad\text{and}\quad 0 < p_\sharp-(1-\varepsilon)(p-1) < p_\sharp,
		\end{equation*}
		therefore~\eqref{eq:1-int-toprove} entails that
		\begin{equation}
			\label{eq:psharp-1-3}
			\int_{B_{2R}} v^{\varepsilon-p_\sharp-1} \,\abs*{\nabla v}^{p} \, dx \leq C R^{\frac{p}{p-1}\varepsilon},
		\end{equation}
		moreover, Proposition~\ref{prop:int-uv-1} yields
		\begin{equation}
			\label{eq:psharp-1-4}
			\int_{B_{2R}} v^{-p_\sharp+(1-\varepsilon)(p-1)} \, dx \leq C R^{(1-\varepsilon)p}.
		\end{equation}
		As a result, combining~\eqref{eq:psharp-1-1}--\eqref{eq:psharp-1-4} immediately implies the validity of~\eqref{eq:int-v-q-psharp-1}.
		
		Finally, since in~\eqref{eq:int-v-q-psharp-1} the constant~$C$ is independent of~$R$, using the definitions of~$v$ and~$p_\sharp$ given in~\eqref{eq:defv-class} and~\eqref{eq:psharp}, respectively, we deduce that~$u \in L^{\past-1}(\R^n)$.
	\end{proof}
	
	We recall that, in Lemma~2.1 of~\cite{vet-plap}, V\'etois proved the following optimal estimate.
	
	\begin{propositionE}
	\label{prop:est-sharp-vet}
		Let~$n \in \N$,~$1<p<n$, and let~$u \in W^{1,p}_{\loc}(\R^n) \cap L^\infty_{\loc}(\R^n)$ be a non-negative, non-trivial, local weak solution to~\eqref{eq:critica-plap}. Suppose that~$r \in [0,p]$ and~$q < p_\sharp +1$, then for every~$R \geq 1$ we have
		\begin{equation*}
			\int_{B_R} v^{-q} \,\abs{\nabla v}^r \, dx \leq  C R^{n-\frac{pq-r}{p-1}} \quad\text{for every } q \in (-\infty,r),
		\end{equation*}
		where~$C>0$ is a constant independent of~$R$.
	\end{propositionE}
	
	In the following result, we extend Proposition~\ref{prop:est-sharp-vet} by providing new optimal or nearly optimal estimates for the same integral quantity.
	
	\begin{proposition}
		\label{prop:est-sharp-grad}
		Let~$n \in \N$,~$1<p<n$, and let~$u \in W^{1,p}_{\loc}(\R^n) \cap L^\infty_{\loc}(\R^n)$ be a non-negative, non-trivial, local weak solution to~\eqref{eq:critica-plap} satisfying~\eqref{eq:ass-inf}. Suppose that
		\begin{equation}
		\label{eq:range-r-rem}
			r \in \left(0,\frac{n}{n-1} (p-1)\right)\!,
		\end{equation}
		and set
		\begin{equation}
		\label{eq:def-qr}
			q_r \coloneqq \frac{np-n+r}{p} > r.
		\end{equation}
		Then, there exist~$R_1 \geq R_0$ -- where~$R_0$ is defined in Proposition~\ref{prop:int-uv-1} -- and~$\varepsilon_1 \in (0,1)$, depending on~$n$,~$p$,~$r$, and~$q$, such that for every~$R \geq R_1$ and~$\varepsilon \in (0,\varepsilon_1)$, we have
		\begin{gather}
		\label{eq:int-tbp-1}
			\int_{B_{R}} v^{-q} \,\abs*{\nabla v}^{r} \, dx \leq C R^{n-\frac{pq-r}{p-1} + \varepsilon} \quad\text{for every } q \in \left(-\infty,q_r \right) \!, \\
		\label{eq:int-tbp-2}
			\int_{B_{R}} v^{-q} \,\abs*{\nabla v}^{r} \, dx \leq C R^{\varepsilon} \quad\text{for every } q \in \left[ q_r, p_\sharp+1 - \frac{n-p}{np\left(p-1\right)} \, r \right) \!,
		\end{gather}
		where~$C>0$ is a constant independent of~$R$. Moreover, for any~$p \in (1,n)$, it holds
		\begin{equation}
		\label{eq:est-grad-p}
			\int_{B_R} v^{-q} \,\abs*{\nabla v}^p \, dx \leq C R^{n-\frac{p}{p-1}(q-1)} \quad\text{for every } q \in (-\infty,p_\sharp+1).			
		\end{equation}
	\end{proposition}
	
	\begin{remark}
	\label{rem:optim-q}
		We emphasize that the exponent~$q_r$, defined in~\eqref{eq:def-qr}, is precisely the threshold ensuring the integrability of~$v^{-q} \,\abs{\nabla v}^r$ for a~$p$-bubble.
		
		Moreover, we note that the range of~$r$ in~\eqref{eq:range-r-rem} is related to both the regularity theory and the general theory of~$p$-superharmonic functions. Indeed, this range appears in the integral estimate for~$\abs*{\nabla u}$ -- see estimate~\eqref{eq:grad-rewrite} below. Since~$u^{\past-1} \in L^1(\R^n)$, Corollary~1 in~\cite{km-guide} yields the Marcinkiewicz estimate~$\abs{\nabla u} \in L^{\frac{n}{n-1} (p-1),\infty}(\R^n)$. Additionally, the extremal exponent in~\eqref{eq:range-r-rem} should be compared with the maximal integrability exponent of the gradient of general~$p$-superharmonic functions -- see~\cite[Theorems~7.45 and~7.46]{hein-npt}.
	\end{remark}	
	
	\begin{proof}[Proof of Proposition~\ref{prop:est-sharp-grad}]
		In the following calculations~$C>0$ will denote a constant independent of~$R$, which may vary from line to line.
		
		Observe that Formula~(2.21) in~\cite{bveron-poho} implies that, for any~$\ell>p-1$ and~$k>1+1/\ell$, we have
		\begin{equation}
			\label{eq:bv-poho-grad}
			\left(\,\dashint_{B_{R}} \,\abs*{\nabla u}^{\frac{p}{k}} \, dx \right)^{\!\!\frac{k}{p}} \leq \frac{C}{R} \left(\,\dashint_{B_{2R}} u^{\ell} \, dx \right)^{\!\!\frac{1}{\ell}}\!.
		\end{equation}
		We choose
		\begin{equation}
			\label{eq:def-ell}
			\ell \coloneqq \frac{p}{n-p} \, p_\sharp = \frac{n(p-1)}{n-p} > p-1,
		\end{equation}
		so that
		\begin{equation*}
			\gamma \coloneqq \frac{p}{k} < \frac{\ell \, p}{1+\ell} = \frac{n}{n-1} (p-1).
		\end{equation*}
		Applying~\eqref{eq:bv-poho-grad}, we deduce that
		\begin{equation}
			\label{eq:grad-ugamma}
			\left(\int_{B_{R}} \,\abs*{\nabla u}^{\gamma} \, dx \right)^{\!\!\frac{1}{\gamma}} \leq C R^{\frac{n}{\gamma}-\frac{n}{\ell}-1} \left(\int_{B_{2R}} u^{\ell} \, dx \right)^{\!\!\frac{1}{\ell}} = R^{\frac{n}{\gamma}-\frac{n-1}{p-1}} \left(\int_{B_{2R}} v^{-p_\sharp} \, dx \right)^{\!\!\frac{1}{\ell}}\!,
		\end{equation}
		where the last identity follows form~\eqref{eq:defv-class} and~\eqref{eq:def-ell}. On the other hand, using Proposition~\ref{prop:int-uv-2}, we infer that
		\begin{equation*}
			\left(\int_{B_{2R}} v^{-p_\sharp} \, dx \right)^{\!\!\frac{1}{\ell}} \leq C R^{\frac{n-p}{n p_\sharp} \varepsilon}
		\end{equation*}
		for every~$\varepsilon \in (0,\varepsilon_0)$. This, together with~\eqref{eq:grad-ugamma}, yields
		\begin{equation}
			\label{eq:grad-rewrite}
			\int_{B_{R}} \,\abs*{\nabla u}^{\gamma} \, dx \leq C R^{n-\frac{n-1}{p-1} \gamma + \varepsilon},
		\end{equation}
		for every~$\varepsilon \in (0,\varepsilon_0)$, up to a redefinition of~$\varepsilon_0>0$.
		
		Since from~\eqref{eq:defv-class} we get that
		\begin{equation}
		\label{eq:gradu-gradv}
			\nabla u = - \frac{n-p}{p} v^{-\frac{n}{p}} \,\nabla v \quad\text{in } \R^n,
		\end{equation}
		estimate~\eqref{eq:grad-rewrite} is equivalent to
		\begin{equation}
			\label{eq:gradv-gamma}
			\int_{B_{R}} v^{-\frac{n}{p} \gamma} \,\abs*{\nabla v}^{\gamma} \, dx \leq C R^{n-\frac{n-1}{p-1} \gamma +  \varepsilon} \quad\text{for every } \gamma \in \left(0,\frac{n}{n-1} (p-1)\right)\!.
		\end{equation}
		Therefore, for 
		\begin{equation*}
			r \in \left(0,\frac{n}{n-1} (p-1)\right) \quad\text{and}\quad q \leq \frac{n}{p} \, r,
		\end{equation*}
		using~\eqref{eq:babv-serzou} and~\eqref{eq:gradv-gamma}, we deduce that
		\begin{equation*}
			\begin{split}
				\int_{B_{R}} v^{-q} \,\abs*{\nabla v}^{r} \, dx &= \int_{B_{R}} v^{\frac{n}{p} r-q} v^{-\frac{n}{p} r} \,\abs*{\nabla v}^{r} \, dx \leq C R^{\frac{p}{p-1}\left(\frac{n}{p} r-q \right)} \int_{B_{R}} v^{-\frac{n}{p} r} \,\abs*{\nabla v}^{r} \, dx \\
				&\leq C R^{n-\frac{n-1}{p-1} r + \frac{p}{p-1}\left(\frac{n}{p} r-q \right) + \varepsilon} = C R^{n-\frac{pq-r}{p-1} + \varepsilon}.
			\end{split}
		\end{equation*}
		
		We now consider the case
		\begin{equation}
			\label{eq:rest-r-q}
			r \in \left(0,\frac{n}{n-1} (p-1)\right) \quad\text{and}\quad q > \frac{n}{p} \, r,
		\end{equation}
		and define, for~$\varepsilon \in (0,1)$,
		\begin{equation}
			\label{eq:def-theta}
			\theta \coloneqq \frac{n}{n-1} \frac{p-1}{r} - \varepsilon.
		\end{equation}

		Clearly, from~\eqref{eq:rest-r-q}, there exists~$\varepsilon_1 >0$, depending only on~$n$,~$p$, and~$r$, such that~$\theta>1$ for all~$\varepsilon \in (0,\varepsilon_1)$. Applying H\"older inequality, we deduce
		\begin{equation}
			\label{eq:int-grad-toest}
			\begin{split}
				\int_{B_{R}} v^{-q} \,\abs*{\nabla v}^{r} \, dx &= \int_{B_{R}} v^{-\frac{n}{p} r} \,\abs*{\nabla v}^{r} v^{\frac{n}{p} r-q} \, dx \\
				&\leq \left(\int_{B_{R}} v^{-\frac{n}{p} r\theta} \,\abs*{\nabla v}^{r\theta} \, dx\right)^{\!\!\frac{1}{\theta}} \left(\int_{B_{R}} v^{-\left(q-\frac{n}{p} r\right) \theta'} dx\right)^{\!\!\frac{1}{\theta'}}\!.
			\end{split}
		\end{equation}
		Thanks to the definition of~$\theta$ in~\eqref{eq:def-theta}, estimate~\eqref{eq:gradv-gamma} immediately yields that
		\begin{equation}
			\label{eq:int-gradtoest-1}
			\int_{B_{R}} v^{-\frac{n}{p} r\theta} \,\abs*{\nabla v}^{r\theta} \, dx \leq C R^{n-\frac{n-1}{p-1} r\theta + \varepsilon}.
		\end{equation}
		To estimate the second integral on the right-hand side of~\eqref{eq:int-grad-toest}, we further distinguish between two cases based on the value of~$q$. First, we consider the case
		\begin{equation*}
			\frac{n}{p} \, r < q < \frac{np-n+r}{p}.
		\end{equation*}
		Observe that this condition is consistent in light of the restriction on~$r$ in~\eqref{eq:rest-r-q}. Possibly taking a smaller~$\varepsilon_1>0$, depending on~$q$ as well, we deduce that
		\begin{equation*}
			\left(q-\frac{n}{p} \, r\right) \theta' < p_\sharp,
		\end{equation*}
		for every~$\varepsilon \in (0,\varepsilon_1)$. From this, Proposition~\ref{prop:int-uv-1} readily implies that
		\begin{equation}
			\label{eq:int-gradtoest-2}
			\int_{B_{R}} v^{-\left(q-\frac{n}{p} r\right) \theta'} dx \leq C R^{n-\frac{p}{p-1} \left(q-\frac{n}{p} r\right) \theta'}.
		\end{equation}
		As a result, combining~\eqref{eq:int-grad-toest}--\eqref{eq:int-gradtoest-2}, we conclude that
		\begin{equation*}
			\int_{B_{R}} v^{-q} \,\abs*{\nabla v}^{r} \, dx \leq C R^{n-\frac{n-1}{p-1} r - \frac{p}{p-1}\left(q-\frac{n}{p} r\right) + r \varepsilon} = C R^{n-\frac{pq-r}{p-1} + \varepsilon},
		\end{equation*}
		up to a redefinition of~$\varepsilon_1>0$. This completes the proof of~\eqref{eq:int-tbp-1}.
		
		Next, we consider the case
		\begin{equation*}
			\frac{np-n+r}{p} \leq q < \frac{np-n+p}{p} - \frac{n-p}{np\left(p-1\right)} r.
		\end{equation*}
		Note that this condition is consistent provided that~$r<p$ and this is verified since~$p<n$ and~\eqref{eq:rest-r-q} is in force. Moreover, it implies that
		\begin{equation*}
			p_\sharp \leq \left(q-\frac{n}{p} \, r\right) \theta' < p_\sharp +1 ,
		\end{equation*}
		for every~$\varepsilon \in (0,\varepsilon_1)$, possibly taking a smaller~$\varepsilon_1>0$, for~$q>q_r$. Hence, Proposition~\ref{prop:int-uv-2} yields that
		\begin{equation}
			\label{eq:int-gradtoest-22}
			\int_{B_{R}} v^{-\left(q-\frac{n}{p} r\right) \theta'} dx \leq C R^{\varepsilon},
		\end{equation}
		for all~$\varepsilon \in (0,\varepsilon_1)$, possibly for a smaller~$\varepsilon_1>0$. Therefore, combining~\eqref{eq:int-grad-toest},~\eqref{eq:int-gradtoest-1}, and~\eqref{eq:int-gradtoest-22}, we conclude that
		\begin{equation*}
			\int_{B_{R}} v^{-q} \,\abs*{\nabla v}^{r} \, dx \leq C R^{\left(n-\frac{n-1}{p-1} r\theta + \varepsilon \right) \frac{1}{\theta} + \frac{\varepsilon}{\theta'}} \leq C R^{\varepsilon},
		\end{equation*}
		where we used the definition of~$\theta$ in~\eqref{eq:def-theta}. 
		
		For~$q=q_r$, using~\eqref{eq:babv-serzou}, we have
		\begin{equation*}
			\int_{B_{R}} v^{-q_r} \,\abs*{\nabla v}^{r} \, dx = \int_{B_{R}} v^\varepsilon v^{-q_r-\varepsilon} \,\abs*{\nabla v}^{r} \, dx \leq C R^{\frac{p}{p-1}\varepsilon} \int_{B_{R}} v^{-q_r-\varepsilon} \,\abs*{\nabla v}^{r} \, dx \leq C R^{\varepsilon},
		\end{equation*}
		possibly for a smaller~$\varepsilon_1>0$. This shows the validity of~\eqref{eq:int-tbp-2}.
		
		Finally, estimate~\eqref{eq:est-grad-p} follows directly from~\eqref{eq:1-int-toprove}.
	\end{proof}
	
	
	\section{Proof of the classification results}
	\label{sec:proof-12-13}
	
	In this section, we prove the classification results in Theorems~\ref{th:class-bounded},~\ref{th:class-Lq}, and~\ref{th:class-crescita}.
	
	\begin{proof}[Proof of Theorems~\ref{th:class-bounded} and~\ref{th:class-Lq}]
		In the following calculations~$C>0$ will denote a constant independent of~$R$, which may vary from line to line.
		
		Since, by assumption and Proposition~\ref{prop:int-uv-2},~$u \in L^{\past-1}(\R^n) \cap L^q(\R^n)$ for some~$q \in [\past,+\infty]$, interpolating, we deduce that~$u \in L^{\past}\!(\R^n)$. Consequently, a standard test function argument yields that~$u \in \mathcal{D}^{1,p}(\R^n)$. Indeed, let~$\eta \in C^\infty_c(\R^n)$ be a cut-off function as in the proof of Proposition~\ref{prop:int-uv-2} and test~\eqref{eq:critica-plap} with~$\psi = u \eta^p$. Using Young's inequality, we obtain
		\begin{equation*}
			\int_{\R^n} \eta^p  \,\abs*{\nabla u}^p \, dx \leq \int_{\R^n} \eta^p u^{\past} dx + C \int_{A_{R,2R}} u^{p} \,\abs*{\nabla\eta}^p \, dx.
		\end{equation*}
		Applying H\"older inequality then gives
		\begin{equation*}
			\int_{B_R} \,\abs*{\nabla u}^p \, dx \leq \int_{\R^n} \eta^p u^{\past} dx + C \,\frac{\abs*{B_{2R}}^{\frac{1}{p}-\frac{1}{\past}}}{R^p} \left(\int_{A_{R,2R}} u^{\past}  dx\right)^{\!\!\frac{p}{\past}} \leq C.
		\end{equation*}
		Therefore, letting~$R \to +\infty$, we deduce that~$\abs*{\nabla u} \in L^p(\R^n)$, and hence~$u \in \mathcal{D}^{1,p}(\R^n)$. The desired conclusion then follows from the classification in~\cite{dm,sciu,vet}.
		
		We now present an alternative proof based solely on the~$P$-function approach, which in particular does not rely on any~$L^\infty$-estimate.
		
		Setting
		\begin{equation*}
			\mathring{W} \coloneqq \nabla\stressu - \frac{\tr\nabla\stressu}{n} \Id,
		\end{equation*}
		in Proposition 2.3 of~\cite{ou}, Ou proved that the~$P$-function defined in~\eqref{eq:defPfunct-class} satisfies
		\begin{multline*}
			\int_{\R^n} v^{1-n} P^m \tr \!\mathring{W}^2 \eta \, dx + nm \int_{\R^n} v^{-n} P^{m-1} \,\abs{\nabla v}^{p-2} \left\langle \mathring{W}^2 \nabla v, \nabla v \right\rangle \eta \, dx \\
			\leq - \int_{\R^n} v^{1-n} P^{m} \,\abs{\nabla v}^{p-2} \left\langle \mathring{W} \nabla v, \nabla\eta \right\rangle dx
		\end{multline*}
		for every non-negative~$\eta \in C^\infty_c(\R^n)$ and~$m \in \R$ -- see also~\cite[Lemma~2.2]{vet-plap}. From this, arguing as in~\cite{ou} -- see also~\cite[Formula~(2.17)]{vet-plap} --, one can see that in order to obtain the classification, it suffices to prove, for~$\varepsilon \in (0,1)$ sufficiently small, that
		\begin{equation}
			\label{eq:int-o(1)-tbp-2}
			\int_{B_R} v^{1-n} P^{-\frac{p-1}{p}+\varepsilon} \,\abs*{\nabla v}^{2(p-1)} \, dx = o(R^2) \quad\text{as } R \to +\infty.
		\end{equation}
		Using the definition of~$P$ in~\eqref{eq:defPfunct-class}, we easily see that
		\begin{equation}
			\label{eq:est-int-o(1)}
			\int_{B_R} v^{1-n} P^{-\frac{p-1}{p}+\varepsilon} \,\abs*{\nabla v}^{2(p-1)} \, dx \leq C \int_{B_R} v^{2-\frac{1}{p}-n-\varepsilon} \,\abs*{\nabla v}^{p - 1 + p \varepsilon} \, dx.
		\end{equation}
		Moreover, we may assume that
		\begin{equation*}
			1<p\leq \frac{n+1}{3},
		\end{equation*}
		since otherwise the conclusion follows directly from~\cite{ou}. This also implies that
		\begin{equation}
			\label{eq:rest-q-class-ener}
			p_\sharp+1 \leq n-2+\frac{1}{p}+\varepsilon < n
		\end{equation}
		for~$\varepsilon \in (0,1)$ sufficiently small.
		
		Next, we observe that
		\begin{equation}
			\label{eq:u-D1p-classif}
			u \in \mathcal{D}^{1,p}(\R^n) \quad\text{if and only if}\quad \int_{\R^n} v^{-n} \, dx + \int_{\R^n} v^{-n} \,\abs*{\nabla v}^p \, dx < +\infty.
		\end{equation}
		As a consequence, interpolating~\eqref{eq:int-v-q-psharp-1} with~\eqref{eq:u-D1p-classif}, we obtain
		\begin{equation}
			\label{eq:int-v-q-very-lagre}
			\int_{B_R} v^{-q} \, dx \leq C \quad\text{for every } q \in [p_\sharp+1,n].
		\end{equation}
		Similarly, interpolating~\eqref{eq:est-grad-p} and~\eqref{eq:u-D1p-classif}, we deduce
		\begin{equation}
			\label{eq:est-grad-p-very-large}
			\int_{B_R} v^{-q} \,\abs*{\nabla v}^p \, dx \leq C R^{\varepsilon} \quad\text{for every } q \in [p_\sharp+1,n),
		\end{equation}
		possibly after redefining~$\varepsilon>0$. Furthermore, for every~$q \in [p_\sharp+1,n)$, choosing~$\lambda \in (0,1)$ such that~$p-1+p\varepsilon = \lambda p$ and interpolating, we obtain
		\begin{equation}
			\label{eq:interp-tocomb}
			\int_{B_R} v^{-q} \,\abs*{\nabla v}^{p-1+p\varepsilon} \, dx \leq \left(\int_{B_R} v^{-q} \,\abs*{\nabla v}^{p} \, dx\right)^{\!\!\lambda} \left(\int_{B_R} v^{-q} \, dx\right)^{\!\!1-\lambda}\!.
		\end{equation}
		Finally, combing~\eqref{eq:int-v-q-very-lagre}--\eqref{eq:interp-tocomb}, we find that
		\begin{equation*}
			\label{eq:interp-combined}
			\int_{B_R} v^{-q} \,\abs*{\nabla v}^{p-1+p\varepsilon} \, dx \leq C R^{\varepsilon}  \quad\text{for every } q \in [p_\sharp+1,n),
		\end{equation*}
		for all sufficiently large~$R \geq 1$, and provided~$\varepsilon \in (0,1)$ is small enough. This estimate, together with~\eqref{eq:est-int-o(1)} and~\eqref{eq:rest-q-class-ener}, ensures the validity of~\eqref{eq:int-o(1)-tbp-2}.
	\end{proof}
	
	\begin{proof}[Proof of Corollary~\ref{cor:absolut-cont}]
		Since, by Proposition~\ref{prop:int-uv-2},~$u^{\past-1} \in L^{1}(\R^n)$ and it is uniformly continuous, it follows that~$u^{p^{\ast}-1}$ vanishes at infinity. Consequently,~$u \in L^{\infty}(\R^n)$. We are now in a position to apply Theorem~\ref{th:class-bounded}.
	\end{proof}
	
	We conclude this section proving Theorem~\ref{th:class-crescita}.
	
	\begin{proof}[Proof of Theorem~\ref{th:class-crescita}]
		
		In the following calculations~$C>0$ will denote a constant independent of~$R$, which may vary from line to line.
		
		As in the second part of the proof of Theorems~\ref{th:class-bounded} and~\ref{th:class-Lq} -- see, in particular,~\eqref{eq:int-o(1)-tbp-2} --, in order to obtain the classification, it suffices to prove that, for~$\varepsilon \in (0,1)$ sufficiently small, we have
		\begin{equation}
			\label{eq:int-o(1)-tbp}
			\int_{B_R} v^{1-n} P^{-\frac{p-1}{p}+\varepsilon} \,\abs*{\nabla v}^{2(p-1)} \, dx = o(R^2) \quad\text{as } R \to +\infty.
		\end{equation}
		We observe that the growth condition~\eqref{eq:growth-beta} implies
		\begin{equation}
			\label{eq:growth-2}
			v(x) \geq C \,\abs*{x}^{-\frac{\beta p}{n-p}} \quad\text{for every } \abs*{x} \geq 1,	
		\end{equation}
		moreover, we set
		\begin{equation*}
			\mathtt{q} \coloneqq p_\sharp+1 - \frac{n-p}{np} - \frac{p(n-1)}{n(p-1)} \,\varepsilon
		\end{equation*}
		and choose~$\varepsilon \in (0,1)$ so small such that
		\begin{equation*}
			p-1+p \varepsilon < \frac{n}{n-1} (p-1) \quad\text{and}\quad \mathtt{q} > q_{p-1+p \varepsilon}.
		\end{equation*}
		Using~\eqref{eq:defPfunct-class}, we obtain
		\begin{equation}
			\label{eq:est-o(1)-1}
			\begin{split}
				\int_{B_R} v^{1-n} P^{-\frac{p-1}{p}+\varepsilon} \,\abs*{\nabla v}^{2(p-1)} \, dx &\leq C \int_{B_R} v^{1-n+\frac{p-1}{p}-\varepsilon} \,\abs*{\nabla v}^{p-1+p \varepsilon} \, dx \\
				&= C \int_{B_R} v^{1-n+\frac{p-1}{p}-\varepsilon+\mathtt{q}} \, v^{-\mathtt{q}} \,\abs*{\nabla v}^{p-1+p \varepsilon} \, dx.
			\end{split}
		\end{equation}
		Since, thanks to~\eqref{eq:cond-p-beta}, it follows that
			\begin{equation*}
				1-n+\frac{p-1}{p}-\varepsilon+\mathtt{q} = \frac{(3n+1)p-(n^2+2n)}{np} - \frac{n(p-1)+p(n-1)}{n(p-1)} \, \varepsilon <0,
			\end{equation*}
		Proposition~\ref{prop:est-sharp-grad},~\eqref{eq:growth-2}, and~\eqref{eq:est-o(1)-1} yield
			\begin{equation*}
				\int_{B_R} v^{1-n} P^{-\frac{p-1}{p}+\varepsilon} \,\abs*{\nabla v}^{2(p-1)} \, dx \leq C R^{\frac{n^2+2n-(3n+1)p}{n(n-p)} \, \beta + c_{n,p} \, \beta \varepsilon},
			\end{equation*}
		for some~$c_{n,p}>0$ and for~$\varepsilon \in (0,1)$ sufficiently small. From this, thanks to~\eqref{eq:cond-p-beta}, we infer the validity of~\eqref{eq:int-o(1)-tbp} and the proof is complete.
	\end{proof}

	
	\appendix
	
	\section{Further integral estimates involving the gradient}
	\label{sec:furt-int-est}
 
	In this appendix, we present new estimates for the integral involving~$\abs{\nabla u}$ which are not used in the proofs of the main results of this paper. However, they generalize and improve some related estimates already available in literature and may be useful for further extending the classification result of Theorem~\ref{th:class-noener}. 

	Specifically, in the following result we enlarge the range of exponents for which we have a nearly sharp integral estimate for~$v^{-q} \,\abs*{\nabla v}^{r}$. 	

	\begin{proposition}
	\label{prop:grad-second}
		Let~$n \in \N$,~$1<p<n$, and let~$u \in W^{1,p}_{\loc}(\R^n) \cap L^\infty_{\loc}(\R^n)$ be a non-negative, non-trivial, local weak solution to~\eqref{eq:critica-plap} satisfying~\eqref{eq:ass-inf}.
		Suppose that~$r \in (0,p)$ and let~$q_r$ be defined in~\eqref{eq:def-qr}.
		Then, there exists~$\varepsilon_2 \in (0,\varepsilon_1)$, depending on~$n$,~$p$,~$r$, and~$q$, such that for every~$R \geq R_1$ -- where~$R_1$ is defined in Proposition~\ref{prop:est-sharp-grad} -- and~$\varepsilon \in (0,\varepsilon_2)$, we have
		\begin{equation}
			\label{eq:int-grad-all-r}
			\int_{B_{R}} v^{-q} \,\abs*{\nabla v}^{r} \, dx \leq C R^{n-\frac{pq-r}{p-1} + \varepsilon} \quad\text{for every } q \in \left(-\infty,q_r \right) \!,
		\end{equation}
		where~$C>0$ is a constant independent of~$R$.
	\end{proposition}
	\begin{proof} 
		In the following calculations~$C>0$ will denote a constant independent of~$R$, which may vary from line to line.
		
		We set
		\begin{equation*}
			\mathtt{r} \coloneqq \frac{n}{n-1} (p-1),
		\end{equation*}
		so that, thanks to~\eqref{eq:int-tbp-1}, estimate~\eqref{eq:int-grad-all-r} holds for every~$r \in (0,\mathtt{r})$.
		
		Thus, we may assume that~$r \in [\mathtt{r},p)$ and choose~$\lambda \in (0,1)$ and~$\varepsilon \in \left(0,\min\{\varepsilon_1,\mathtt{r}\}\right)$ -- where~$\varepsilon_1$ is defined in Proposition~\ref{prop:est-sharp-grad} -- such that~$r = \lambda (\mathtt{r}-\varepsilon) + (1-\lambda) p$. Coherently, we write~$q = \lambda q_{\mathtt{r}-\varepsilon} + (q-\lambda q_{\mathtt{r}-\varepsilon})$, recalling~\eqref{eq:def-qr}. Therefore, using H\"older inequality, we get
		\begin{equation}
		\label{eq:est-grad-gener}
			\int_{B_{R}} v^{-q} \,\abs*{\nabla v}^{r} \, dx \leq \left(\int_{B_{R}} v^{-q_{\mathtt{r}-\varepsilon}} \,\abs*{\nabla v}^{\mathtt{r}-\varepsilon} \, dx\right)^{\!\!\lambda} \left(\int_{B_{R}} v^{-\frac{q-\lambda q_{\mathtt{r}-\varepsilon}}{1-\lambda}} \,\abs*{\nabla v}^{p} \, dx\right)^{\!\! 1-\lambda} \!.
		\end{equation}
		We now estimate both terms on the right-hand side of~\eqref{eq:est-grad-gener}. For the first integral, Proposition~\ref{prop:est-sharp-grad} yields
		\begin{equation}
		\label{eq:est-grad-gener-1}
			\int_{B_{R}} v^{-q_{\mathtt{r}-\varepsilon}} \,\abs*{\nabla v}^{\mathtt{r}-\varepsilon} \, dx \leq C R^{n-\frac{pq_{\mathtt{r}-\varepsilon}-\mathtt{r}+\varepsilon}{p-1} + \varepsilon},
		\end{equation}
		for every~$\varepsilon \in (0,\varepsilon_2)$. Moreover, since~$q<q_r$, it follows that
		\begin{equation*}
			\frac{q-\lambda q_{\mathtt{r}-\varepsilon}}{1-\lambda} < p_\sharp +1,
		\end{equation*}
		therefore, from~\eqref{eq:est-grad-p}, we deduce that
		\begin{equation}
		\label{eq:est-grad-gener-2}
			\int_{B_{R}} v^{-\frac{q-\lambda q_{\mathtt{r}-\varepsilon}}{1-\lambda}} \,\abs*{\nabla v}^{p} \, dx \leq C R^{n-\frac{p}{p-1}\left(\frac{q-\lambda q_{\mathtt{r}-\varepsilon}}{1-\lambda}-1\right)}.
		\end{equation}
		Combining~\eqref{eq:est-grad-gener}--\eqref{eq:est-grad-gener-2} and exploiting the previously defined splitting of~$r$ to simplify the calculations, we conclude that
		\begin{equation*}
			\int_{B_{R}} v^{-q} \,\abs*{\nabla v}^{r} \, dx \leq \left(R^{n-\frac{pq_{\mathtt{r}-\varepsilon}-\mathtt{r}+\varepsilon}{p-1} + \varepsilon}\right)^{\!\!\lambda} \left(R^{n-\frac{p}{p-1}\left(\frac{q-\lambda q_{\mathtt{r}-\varepsilon}}{1-\lambda}-1\right)}\right)^{\!\! 1-\lambda} = C R^{n-\frac{pq-r}{p-1} + \varepsilon} ,
		\end{equation*}
 		up to a redefinition of~$\varepsilon>0$ and~$\varepsilon_2>0$. This concludes the proof of~\eqref{eq:int-grad-all-r}.
	\end{proof}
	
	Propositions~\ref{prop:est-sharp-vet},~\ref{prop:est-sharp-grad}, and~\ref{prop:grad-second} give a quite complete picture of the sharp or nearly sharp integral estimates for~$v^{-q} \,\abs{\nabla v}^r$. Apparently, only the case~$q \geq q_r$ is missing.
	
	In this direction, we obtain the following result, which is inspired by Lemmas~2.1,~2.2, and~2.3 in~\cite{bgv}.
	
	\begin{proposition}
	\label{prop:grad-third}
		Let~$n \in \N$,~$1<p<n$, and let~$u \in W^{1,p}_{\loc}(\R^n) \cap L^\infty_{\loc}(\R^n)$ be a non-negative, non-trivial, local weak solution to~\eqref{eq:critica-plap} satisfying~\eqref{eq:ass-inf}. Suppose that
		\begin{equation}
		\label{eq:ass-p-r-q}
			r \in \left[\frac{n}{n-1} (p-1),p-1+\frac{p}{n}\right) \quad\text{and}\quad q \geq r,
		\end{equation}
		and set
		\begin{equation*}
			\mathsf{q}_r \coloneqq \frac{n}{p} \, r > r.
		\end{equation*}
		Then, there exists~$\varepsilon_3 \in (0,\varepsilon_0)$, depending on~$n$,~$p$,~$r$, and~$q$, such that for every~$R \geq R_1$ -- where~$R_1$ is defined in Proposition~\ref{prop:est-sharp-grad} -- and~$\varepsilon \in (0,\varepsilon_3)$, we have	
		\begin{gather}
			\label{eq:int-tbp-1'}
			\int_{B_{R}} v^{-q} \,\abs*{\nabla v}^{r} \, dx \leq C R^{\frac{p}{p-1}\left(\frac{n}{p} r-q \right) + \varepsilon} \quad\text{for every } q \in \left[r,\mathsf{q}_r\right] \!, \\
			\label{eq:int-tbp-2'}
			\int_{B_{R}} v^{-q} \,\abs*{\nabla v}^{r} \, dx \leq C R^{\varepsilon} \quad\text{for every } q \in \left(\mathsf{q}_r, p_\sharp+1 \right) \!,
		\end{gather}
		where~$C>0$ is a constant independent of~$R$.
	\end{proposition}
	\begin{proof}
		In the following calculations~$C>0$ will denote a constant independent of~$R$, which may vary from line to line.
		
		Let~$m \in (0,1)$ be fixed and define~$\phi: [0+\infty) \to [0+\infty)$ by
		\begin{equation*}
			\phi(t) \coloneqq
				\begin{cases}
					\begin{aligned}
						& (1-m) \int_{0}^{t} s^{-m} \, ds		&& \text{if } t \in [0,1], \\
						& 1 + m  \int_{1}^{t} s^{-(1+m)} \, ds	&& \text{if } t \in (1,+\infty).
					\end{aligned}
				\end{cases}			
		\end{equation*}
		Then,~$\phi$ is piecewise smooth, except for a corner at~$t = 1$,~$\phi(0)=0$, and~$0 \leq \phi \leq 2$. Let~$\eta \in C^\infty_c(\R^n)$ be a cut-off function as in the proof of Proposition~\ref{prop:int-uv-2}. We test~\eqref{eq:critica-plap} with~$\psi = \phi(u) \eta$ and observe that
		\begin{equation*}
			\nabla \psi = \phi(u) \nabla\eta + \left((1-m) u^{-m} \chi_{\{u \leq 1\}} + m u^{-(1+m)} \chi_{\{u>1\}}\right) \eta \,\nabla u.
		\end{equation*}
		Hence, we deduce
		\begin{equation*}
			\int_{B_{2R}} \phi'(u) \eta \,\abs*{\nabla u}^p + \phi(u) \left\langle \stressu , \nabla \eta \right\rangle dx = \int_{B_{2R}} u^{\past-1} \phi(u) \eta \, dx,
		\end{equation*}
		from which
		\begin{equation*}
			\int_{B_R} (1-m) \frac{\abs*{\nabla u}^p}{u^m} \chi_{\{u \leq 1\}} + m \,\frac{\abs*{\nabla u}^p}{u^{1+m}} \chi_{\{u>1\}} \, dx \leq 2 \int_{B_{2R}} u^{\past-1} \, dx + \frac{2}{R} \int_{A_{R,2R}} \,\abs*{\nabla u}^{p-1} \, dx,
		\end{equation*}
		where we used the definition of~$\eta$. Recall here that, since~$u$ is non-trivial, we must have~$u>0$ in~$\R^n$. The latter, thanks to Proposition~\ref{prop:int-uv-2} and~\eqref{eq:grad-rewrite}, entails  
		\begin{equation}
		\label{eq:gradu-su-u}
			\int_{B_R} (1-m) \frac{\abs*{\nabla u}^p}{u^m} \chi_{\{u \leq 1\}} + m \,\frac{\abs*{\nabla u}^p}{u^{1+m}} \chi_{\{u>1\}} \, dx \leq C R^{\varepsilon},
		\end{equation}
		for every~$\varepsilon \in (0,\varepsilon_0)$.
		
		Fix~$0 \leq \mu \leq m \leq M$ and observe that~$u^\mu \geq u^m$ in~$\{0<u \leq 1\}$ and~$u^{1+m} \leq u^{1+M}$ in~$\{u>1\}$. Thus, from~\eqref{eq:gradu-su-u}, we obtain
		\begin{equation}
		\label{eq:gradu-su-u-2}
			\int_{B_R} \frac{\abs*{\nabla u}^p}{u^\mu} \chi_{\{u \leq 1\}} + \frac{\abs*{\nabla u}^p}{u^{1+M}} \chi_{\{u>1\}} \, dx \leq C R^{\varepsilon},
		\end{equation}
		for every~$\varepsilon \in (0,\varepsilon_0)$.
		
		Assume that
		\begin{equation}
		\label{eq:p-gamma-grad}
			\gamma \in \left[\frac{n}{n-1} (p-1),p-1+\frac{p}{n}\right)\!.
		\end{equation}
		Applying H\"older inequality, for~$s>1$, it follows that
		\begin{equation}
		\label{eq:grad-comb-1}
			\int_{B_R} \,\abs*{\nabla u}^\gamma \, dx \leq \left(\int_{B_R} \,\frac{\abs*{\nabla u}^p}{u^s} \, dx\right)^{\!\!\frac{\gamma}{p}} \left(\int_{B_R} u^{\frac{s\gamma}{p-\gamma}} \, dx\right)^{\!\!\frac{p-\gamma}{p}}\!.
		\end{equation}
		Moreover, we have that
		\begin{align*}
			\int_{B_R} \,\frac{\abs*{\nabla u}^p}{u^s} \, dx &= \int_{B_R} \,\frac{\abs*{\nabla u}^p}{u^s} \chi_{\{u \leq 1\}} + \frac{\abs*{\nabla u}^p}{u^s} \chi_{\{u>1\}} \, dx \\
			&= \int_{B_R} \,\frac{\abs*{\nabla u}^p}{u^{1-\varepsilon}} u^{1-\varepsilon-s} \chi_{\{u \leq 1\}} + \frac{\abs*{\nabla u}^p}{u^s} \chi_{\{u>1\}} \, dx \\
			&\leq C R^{(s-1+\varepsilon)\frac{n-p}{p-1}} \int_{B_R} \,\frac{\abs*{\nabla u}^p}{u^{1-\varepsilon}} \chi_{\{u \leq 1\}} + \frac{\abs*{\nabla u}^p}{u^s} \chi_{\{u>1\}} \, dx,
		\end{align*}
		where we used~\eqref{eq:bbu-serzou} for the last inequality, since~$s>1$. Hence,~\eqref{eq:gradu-su-u-2} yields
		\begin{equation}
		\label{eq:grad-comb-2}
			\int_{B_R} \,\frac{\abs*{\nabla u}^p}{u^s} \, dx \leq C R^{(s-1+\varepsilon)\frac{n-p}{p-1}+ \varepsilon}.
		\end{equation}
		We now choose~$s=1+\varepsilon$. Thanks to~\eqref{eq:p-gamma-grad}, there exists~$\varepsilon_3 \in (0,\varepsilon_0)$ such that
		\begin{equation*}
			p_\ast-1 \leq \frac{s\gamma}{p-\gamma} \leq \past-1,
		\end{equation*}
		for every~$\varepsilon \in (0,\varepsilon_3)$. As a consequence, from Proposition~\ref{prop:int-uv-2}, we infer that
		\begin{equation}
		\label{eq:grad-comb-3}
			\int_{B_R} u^{\frac{s\gamma}{p-\gamma}} \, dx \leq C R^{\varepsilon}.
		\end{equation}
		Combining~\eqref{eq:grad-comb-1}--\eqref{eq:grad-comb-3}, we conclude that
		\begin{equation}
		\label{eq:est-gradu-gamma-large}
			\int_{B_R} \,\abs*{\nabla u}^\gamma \, dx \leq C R^{\varepsilon} \quad\text{for every } \gamma \in \left[\frac{n}{n-1} (p-1),p-1+\frac{p}{n}\right)\!,
		\end{equation}
		possibly taking a smaller~$\varepsilon_3>0$ and up to a redefinition of~$\varepsilon>0$.
		
		With the help of~\eqref{eq:gradu-gradv}, we see that~\eqref{eq:est-gradu-gamma-large} immediately implies
		\begin{equation}
		\label{eq:est-gradv-gamma-large}
			\int_{B_{R}} v^{-\frac{n}{p} \gamma} \,\abs*{\nabla v}^{\gamma} \, dx \leq C R^{\varepsilon} \quad\text{for every } \gamma \in \left[\frac{n}{n-1} (p-1),p-1+\frac{p}{n}\right)\!.
		\end{equation}
			Therefore, for
		\begin{equation*}
			r \in \left[\frac{n}{n-1} (p-1),p-1+\frac{p}{n}\right) \quad\text{and}\quad r \leq q \leq \frac{n}{p} \, r,
		\end{equation*}
		using~\eqref{eq:babv-serzou} and~\eqref{eq:est-gradv-gamma-large}, we deduce that
		\begin{equation*}
			\begin{split}
				\int_{B_{R}} v^{-q} \,\abs*{\nabla v}^{r} \, dx &= \int_{B_{R}} v^{\frac{n}{p} r-q} v^{-\frac{n}{p} r} \,\abs*{\nabla v}^{r} \, dx \leq C R^{\frac{p}{p-1}\left(\frac{n}{p} r-q \right)} \int_{B_{R}} v^{-\frac{n}{p} r} \,\abs*{\nabla v}^{r} \, dx \\
				&\leq C R^{\frac{p}{p-1}\left(\frac{n}{p} r-q \right) + \varepsilon}.
			\end{split}
		\end{equation*}
		This proves the validity of~\eqref{eq:int-tbp-1'}.
		
		We now consider the case
		\begin{equation}
			\label{eq:rest-r-q-1}
			r \in  \left[\frac{n}{n-1} (p-1),p-1+\frac{p}{n}\right) \quad\text{and}\quad q_r \leq \frac{n}{p} \, r < q < p_\sharp+1,
		\end{equation}
		and define, for~$\varepsilon \in (0,1)$,
		\begin{equation}
		\label{eq:def-theta-1}
			\theta \coloneqq \left(p-1+\frac{p}{n}\right) \frac{1}{r} - \varepsilon.
		\end{equation}
		Clearly, from~\eqref{eq:rest-r-q-1}, we have~$\theta>1$ for all~$\varepsilon \in (0,\varepsilon_3)$, possibly taking a smaller~$\varepsilon_3>0$. Applying H\"older inequality, we deduce
		\begin{equation}
			\label{eq:int-grad-toest-1}
			\begin{split}
				\int_{B_{R}} v^{-q} \,\abs*{\nabla v}^{r} \, dx &= \int_{B_{R}} v^{-\frac{n}{p} r} \,\abs*{\nabla v}^{r} v^{\frac{n}{p} r-q} \, dx \\
				&\leq \left(\int_{B_{R}} v^{-\frac{n}{p} r\theta} \,\abs*{\nabla v}^{r\theta} \, dx\right)^{\!\!\frac{1}{\theta}} \left(\int_{B_{R}} v^{-\left(q-\frac{n}{p} r\right) \theta'} dx\right)^{\!\!\frac{1}{\theta'}}\!.
			\end{split}
		\end{equation}		
		Thanks to the definition of~$\theta$ in~\eqref{eq:def-theta-1}, estimate~\eqref{eq:est-gradv-gamma-large} immediately yields that
		\begin{equation}
			\label{eq:int-grad-toest-2}
			\int_{B_{R}} v^{-\frac{n}{p} r\theta} \,\abs*{\nabla v}^{r\theta} \, dx \leq C R^{\varepsilon}.
		\end{equation}
		To estimate the second integral on the right-hand side of~\eqref{eq:int-grad-toest-1}, we further distinguish between two cases based on the value of~$q$. First, we assume that
		\begin{equation*}
			\mathtt{q} \coloneqq p_\sharp +\frac{n}{np-n+p} \, r < q < p_\sharp+1.
		\end{equation*}
		Note that this condition is consistent thanks to the restriction on~$r$ in~\eqref{eq:ass-p-r-q}. For~$q>\mathtt{q}$, possibly taking a smaller~$\varepsilon_3>0$, depending on~$q$ as well, we can ensure that
		\begin{equation*}
			p_\sharp \leq \left(q-\frac{n}{p} \, r\right) \theta' \leq p_\sharp +1,
		\end{equation*}
		for every~$\varepsilon \in (0,\varepsilon_3)$. From this, Proposition~\ref{prop:int-uv-2} entails
		\begin{equation}
			\label{eq:int-grad-toest-3}
			\int_{B_{R}} v^{-\left(q-\frac{n}{p} r\right) \theta'} dx \leq C R^{\varepsilon},
		\end{equation}
		for all~$\varepsilon \in (0,\varepsilon_3)$, possibly for a smaller~$\varepsilon_3>0$. Therefore, combining~\eqref{eq:int-grad-toest-1}--\eqref{eq:int-grad-toest-3}, we infer the validity of~\eqref{eq:int-tbp-2'} for~$q>\mathtt{q}$.
		
		Finally, we suppose that
		\begin{equation*}
			\mathsf{q}_r < q \leq \mathtt{q}
		\end{equation*}
		and observe again that this condition is consistent thanks~\eqref{eq:ass-p-r-q}. Moreover, by virtue of~\eqref{eq:int-tbp-1'} and~\eqref{eq:int-tbp-2'} for~$q=\mathtt{q}+\varepsilon$, we have
		\begin{equation}
		\label{eq:int-qr-qr}
			\int_{B_{R}} v^{-\mathsf{q}_r} \,\abs*{\nabla v}^{r} \, dx \leq C R^{\varepsilon} \quad\text{and}\quad \int_{B_{R}} v^{-\mathtt{q}-\varepsilon} \,\abs*{\nabla v}^{r} \, dx \leq C R^{\varepsilon}.
		\end{equation}
		Let~$\lambda \in (0,1)$ be such that~$q = \lambda \mathsf{q}_r + (1-\lambda) (\mathtt{q}+\varepsilon)$. Interpolating and employing~\eqref{eq:int-qr-qr}, we get
		\begin{equation*}
			\int_{B_{R}} v^{-q} \,\abs*{\nabla v}^{r} \, dx \leq \left(\int_{B_{R}} v^{-\mathsf{q}_r} \,\abs*{\nabla v}^{r} \, dx\right)^{\!\!\lambda} \left(\int_{B_{R}} v^{-\mathtt{q}-\varepsilon} \,\abs*{\nabla v}^{r} \, dx\right)^{\!\! 1-\lambda} \leq  C R^{\varepsilon}.
		\end{equation*}
		This completes the proof of~\eqref{eq:int-tbp-2'}.
	\end{proof}
	
	\begin{remark}
	\label{rem:q-conf-sub}
		Note that, under the assumption~\eqref{eq:ass-p-r-q}, we have~$\mathsf{q}_r \geq q_r$ which indicate that the range in~\eqref{eq:int-tbp-2'} is not optimal, in view of Remark~\ref{rem:optim-q}. Furthermore, estimate~\eqref{eq:int-tbp-1'} is also suboptimal, suggesting that the bounds in Proposition~\ref{prop:grad-third} could potentially be sharpened.
	\end{remark}
	
	
	\section*{Acknowledgments} 
	\noindent The authors are members of the “Gruppo Nazionale per l'Analisi Matematica, la Probabilità e le loro Applicazioni” (GNAMPA) of the “Istituto Nazionale di Alta Matematica” (INdAM, Italy) and have been partially supported by the “INdAM - GNAMPA Project”, CUP \#E5324001950001\# and by the Research Project of the Italian Ministry of University and Research (MUR) PRIN 2022 “Partial differential equations and related geometric-functional inequalities”, grant number 20229M52AS\_004.
	
	The authors thank Xiaohan Cai for a question that revealed a small flaw in an earlier version of the manuscript.


\end{document}